\documentclass[a4paper]{article}
\usepackage{amsmath}
\usepackage{amsthm}
\usepackage{amsfonts}
\usepackage{amssymb}
\usepackage{graphicx}
\usepackage{authblk}

\newtheorem{theorem}{Theorem}[section]

\newtheorem{example}{Example}[section]
\newtheorem{corollary}{Corollary}[section]

\newtheorem{remark}{Remark}[section]

\begin{document}
\title{{\bf Some results on integer solutions of quadratic polynomials in two variables}}
\author{B. M. Cerna Magui\~na
\\Facultad de Ciencias, Departamento Acad\'emico de Matem\'atica   
\\Universidad Nacional Santiago Ant\'unez de Mayolo
\\Campus Shancayan, Av. Centenario 200, Huaraz, Per\'u 
\\email:bibianomcm@unasam.edu.pe}

\maketitle
\begin{abstract}
In this work we prove some theorems that allow us to find integer solutions of quadratic equations in two variables that represent a natural number.\newline
\end{abstract}
\newpage
\section{Introduction}
It is a non trivial task to determine whether a number is prime or not. The question if a number is prime or not has always attracted mathematicians and number theorists who have obtained some partial successes; but no one has yet obtained an exact mathematical result and acceptable algorithmic structure that solves this question for all or an indefinite set of prime numbers.

In our previous contribution \cite{LS}, we have discussed quadratic polynomials in two variables that represent natural numbers $N$ that end in 1. In this work we improve the results of the article \cite{LS},  in such a way that in the present contribution we have decreased (in computational terms) the number of steps to determine the entire solution of the quadratic polynomial that represents the natural number $N$ that ends in $1$.

We know that to determine the primality of $N$ it is enough to prove that the number is not divisible by prime numbers less than $\sqrt{N}$. Wilson's Theorem \cite{N} allows us to determine exactly when a number is prime: $N\in\mathbb{N}$ is prime if and only if $(N-1)!+1$ is multiple of $N$. The problem of Wilson's method comes when we have larger numbers since the factorial becomes very large, so it is definitely not a practical method.

In this work we use elementary methods to offer alternative methods to determine the primality of any natural number.
\section{Some natural numbers ending in 1 and quadratic polynomials}
\begin{theorem}\label{th1}																										
Let $F:\mathbb{R}^{2}\longrightarrow \mathbb{R}$ be a linear functional defined by $F(x,y)=100ABx+90(A+B)y$, where 
$(A,B)\in\mathbb{N}\times \mathbb{N}$.  Consider $p-81\in\mathbb{N}$ and $p$ a natural number ending in 1. Then, there exist $(x_0,y_0)\in\mathbb{R}^{2}$ and $(M,N)\in\mathbb{N}\times \mathbb{N}$ with $M$ and $N$ being relatively prime numbers such that
$$100ABx_0+90(A+B)y_0=p-81,$$
where
$$AB=\frac{(p-81)(M+N)}{200M},\,\,\, A+B=\frac{(p-81)(M-N)}{180M}\,\,\,\mbox{with}\,\,\, N<M, \frac{N}{M}<1-\frac{36}{\sqrt{p}+9}$$
or
$$AB=\frac{(p-81)(N+M)}{200N},\,\,\, A+B=\frac{(p-81)(N-M)}{180N}\,\,\,\mbox{with}\,\,\, M<N, \frac{M}{N}<1-\frac{36}{\sqrt{p}+9}.$$
\end{theorem}
\begin{proof}
It is clear that there exists $(x_0,y_0)\in\mathbb{R}^{2}$ such that $F(x_0,y_0)=p-81$ since $F$ is surjective. Since
$kerF=\{(9(A+B), -10AB)\}$ and $(kerF)^{\perp}=\{(10AB, 9(A+B))\}$, then for $(x_0,y_0)\in\mathbb{R}^{2}$ we have 
$$(x_0,y_0)=\lambda_1(9(A+B),-10AB)+\lambda_2(10AB,9(A+B)).$$
Therefore
\begin{equation}\label{e11}																													
    F(x_0,y_0)=p-81=10\lambda_2(10^{2}(AB)^{2}+9^{2}(A+B)^{2}), \,\,\,\lambda_2\in\mathbb{Q}.
\end{equation}
Let 
\begin{equation}\label{e15}																													
\begin{split}
100AB& = \frac{p-81+q}{2}\\
90(A+B)& = \frac{p-81-q}{2}
\end{split}
\end{equation}
then from (\ref{e11}) and (\ref{e15}) we have
\begin{equation}\label{e12}
20(p-81)=\lambda_2\left[(p-81)^{2}+q^{2}\right].
\end{equation}
As $\lambda_2\in\mathbb{Q}$, let $\lambda_2=\frac{m}{n}$; $m,n$ relatively prime numbers, then we have the following quadratic polynomial
$$(p-81)^{2}-20\frac{n(p-81)}{m}+q^{2}=0;$$
defining $q =10R$ one has that the general solution is
\begin{equation}\label{e13}
    p-81=\frac{20n\pm\sqrt{400n^{2}-400R^{2}m^{2}}}{2m}
\end{equation}
Using Fermat's last theorem \cite{S} for case $2$ there is $k\in\mathbb{Z}$ such that 
$$n^{2}-R^{2}m^{2}=k^{2}.$$
Denoting $k=M^{2}-N^{2}$, $Rm=2MN$, $n=M^{2}+N^{2}$ and replacing these relationships in (\ref{e13}), we have
\begin{equation}\label{e14}
    p-81=\frac{M}{N}q\,\,\,\,\mbox{or}\,\,\,\, p-81=\frac{N}{M}q.
\end{equation}
Replacing (\ref{e14}) in (\ref{e15}), we have
$$AB=\frac{(p-81)(M+N)}{200M},\,\,\, A+B=\frac{(p-81)(M-N)}{180M}\,\,\,\mbox{with}\,\,\, N<M,$$
or
$$AB=\frac{(p-81)(N+M)}{200N},\,\,\, A+B=\frac{(p-81)(N-M)}{180N}\,\,\,\mbox{with}\,\,\, M<N.$$
Since
$(x_0,y_0)\in\mathbb{R}^{2}$, for $x_0=y_0=1$, we have 
$100AB+90(A+B)=p-81$ which is equivalent to $(10A+9)(10B+9)=p$; so, $p$ has been factorized. In this case $A+B\geq \frac{\sqrt{p}-9}{5}$ since  
$(A,B)$ and $(B,A)$ are solutions of the equation $(10x+9)(10y+9)=p$. Therefore
$\frac{N}{M}<1-\frac{36}{\sqrt{p}+9}$ or $\frac{M}{N}<1-\frac{36}{\sqrt{p}+9}$.
\end{proof}
\begin{example}
Consider the number $p=900071$. So, taking into account the above relationships one has
\begin{equation}\label{e16}
\begin{split}
AB& = 4500+4500\frac{N}{M}-\frac{1}{20}-\frac{N}{20M}\\
A+B& = 5000-5000\frac{N}{M}-\frac{1}{18}+\frac{N}{18M}\\
AB+A+B& = 9500-500\frac{N}{M}-\frac{19}{180}+\frac{N}{180M}.
\end{split}
\end{equation}
For $A \geq 11$ and $A<B$ from the Theorem 2 one gets (see also \cite{LS})
\begin{eqnarray}
0,638969669 \leq \frac{N}{M} \leq 0,9624107255.  \label{uneq}  
\end{eqnarray}
Let  us define $\tau$ as 
\begin{eqnarray}
\tau = 500 \frac{N}{M} + \frac{19}{180}-\frac{1}{180}\frac{N}{M},\,\,\,\,\,\,\tau \in [320,481]. \label{tau1}
\end{eqnarray}
From the equations (\ref{e16}), (\ref{uneq}) and (\ref{tau1}) one gets
\begin{eqnarray}
 191 \leq A+B \leq 1801,\,\,\, 7379 \leq AB \leq 8828,\,\,\,\, 9020 \leq (A+1)(B+1) \leq 9181. \label{uneq1}
\end{eqnarray}
In addition, from (\ref{e16}) and (\ref{tau1}) one can get 
\begin{eqnarray}
 AB + A+B = 9500-\tau,\,\,\, AB = 4499 + 9 \tau\,\,\,\, A+B = 5001-10 \tau,\,\,\,\tau >0. \label{uneq2}
\end{eqnarray}
From (\ref{uneq1}) one has
\begin{eqnarray}
 \frac{7379}{B}\leq A \leq \frac{8828}{B}\,\,\,\,  \mbox{and}  \,\,\,\,\frac{9020}{B+1} -1 \leq A \leq \frac{9181}{B+1}-1. \label{uneq3}
\end{eqnarray}
The last inequalities possess a solution provided that
\begin{eqnarray}
 191 \leq A+B \leq 352\,\,\,\,  \mbox{or}  \,\,\,\, 352 \leq A +B \leq 1801. \label{uneq4}
\end{eqnarray}
Using (\ref{uneq2}) the last relationships provide
\begin{eqnarray}
 465 \leq \tau \leq 481\,\,\,\,  \mbox{or}  \,\,\,\, 320 \leq \tau \leq 464. \label{12e}
\end{eqnarray}
Moreover, $11\leq A \leq 94$  and $94\leq B \leq 755$ imply  $105 \leq A+B \leq 849$. So, these last relationships together with (\ref{uneq2}) and (\ref{uneq4}) allow us to write
\begin{eqnarray}
   \tau \in [465, 481]\,\,\,\,  \mbox{or}  \,\,\,\, \tau \in [416, 464]. \label{13e}
\end{eqnarray}
Therefore, since $\tau$ is multiple of $3$ for $\tau \in [465, 481]$ there exist six possibilities; whereas for $\tau \in [416, 464]$ there exist sixteen possibilities.
\end{example}
\begin{corollary}
Let $p=(10A+9)(10B+9)$ where $(A,B)\in\mathbb{N}\times \mathbb{N}$, then there exist 
$(\lambda_{1},\lambda_{2})\in \mathbb{Q}\times\mathbb{Q}$ such that 
$$\left(\lambda_{2}-\frac{10}{p-81}\right)^{2}+\lambda^{2}_{1}=\left(\frac{10}{p-81}\right)^{2}$$
with
$$\lambda_{2}=\frac{20M^{2}}{(N^{2}+M^{2})(p-81)} \mbox{ and }\lambda_{1}=\frac{-20NM}{(N^{2}+M^{2})(p-81)}$$
where $N$ and $M$ are relatively prime numbers, with  
$$M=\frac{p-81}{10}\,\,\mbox{ or }\,\, M=\frac{p-81}{20},\,\,\,\,\,( N < M).$$
\end{corollary}
\begin{proof}
From the theorem (\ref{th1}) we have
$$\left\|(1,1)\right\|^{2}=\left\|\lambda_{1}(9(A+B),-10AB)+\lambda_{2}(10AB,9(A+B))\right\|^{2}.$$
From  this relationship, the equation  (\ref{e11}) and the Pythagoras' theorem we have
$$2=\left(\lambda^{2}_{1}+\lambda^{2}_{2}\right)\left(9^{2}(A+B)^{2}+10^{2}(AB)^{2}\right)
=\left(\lambda^{2}_{1}+\lambda^{2}_{2}\right)\frac{p-81}{10\lambda_{2}}.$$
This last relation will be used to obtain the result we are searching for.\\
Next, from the relation 
$$(1,1)=\lambda_{1}\left(9(A+B),-10AB\right)+\lambda_{2}(10AB,9(A+B))$$
we have 
$$\frac{20M}{p-81}=M\lambda_{1}-N\lambda_{2}\mbox{ and }M\lambda_{1}=-N\lambda_{2}.$$
And, from the relations 
\begin{eqnarray}
\label{rel1}
m=\frac{20M^{2}}{p-81}\,\,\,\,\mbox{ and }\,\,\,\,AB=(\frac{p-81}{200}) \, (\frac{M+N}{M})
\end{eqnarray}
we have
$$\frac{20M}{p-81}=M\lambda_{1}-N\lambda_{2}\mbox{ and }M\lambda_{1}=-N\lambda_{2}.$$
By solving this system of equations we can have the values of $\lambda_{1}$ and $\lambda_{2}$.\\
Since
$$\lambda_2=\frac{m}{n}, $$  and taking  $n=M^2+N^2$ we get
$$ m=\frac{20M^{2}}{p-81}.$$
Then, this relation together with  the second relationship in (\ref{rel1}) provide us
$$M=\frac{p-81}{10}\,\,\mbox{ or }\,\,M=\frac{p-81}{20}$$
\end{proof}
\begin{corollary}
Let $M = \frac{p-81}{10}$. If $p=(10A+9)(10B+9)$ and $(A,B)\in\mathbb{N}\times \mathbb{N}$ then  $$(p+243-10N)^{2}-36^{2}p$$
has an exact square root.
Similar result holds for
$M=\frac{p-81}{20}.$
\end{corollary}
\begin{proof}
We use the previous corollary  and the relations in Theorem (\ref{th1}) for $AB$ and $A+B$.
\end{proof}
\begin{remark}\label{r1}
If $p$ is a natural number ending in $1$ and there exist $(A,B)\in\mathbb{N}\times \mathbb{N}$ such that $p=(10A+1)(10B+1)$, then, following the steps of Theorem 2.1, one gets the two cases:
The case $N<M$ 
$$AB=\frac{(p-1)(M+N)}{200M},\,\,\, A+B=\frac{(p-1)(M-N)}{20M}$$
with
$$ \frac{N}{M} < 1- \frac{4}{\sqrt{p}+1}$$
The case $N>M$
 $$AB=\frac{(p-81)(M+N)}{200N},\,\,\, A+B=\frac{(p-1)(N-M)}{20N}$$
with
$$ \frac{M}{N} < 1- \frac{4}{\sqrt{p}+1}$$
\end{remark}

\begin{remark}\label{r2}
If $p$ is a natural number ending in $1$ and there exist $(A,B)\in\mathbb{N}\times \mathbb{N}$ such that $p=(10A+3)(10B+7)$, then, following the Theorem 2.1, one gets the next two cases:\\
The case $N<M$ 
$$AB=\frac{(p-21)(M+N)}{200M},\,\,\, 7A+3B=\frac{(p-21)(M-N)}{20M}$$
$$A+B \geq \frac{\sqrt{p+4}}{5}-1\,\,\,\ (\mbox{or} \, \,\, A+B \leq \frac{\sqrt{p+4}}{5}-1)$$
The case $N>M$ 
$$AB=\frac{(p-21)(M+N)}{200N},\,\,\, 7A+3B=\frac{(p-21)(N-M)}{20N}$$
$$A+B \geq \frac{\sqrt{p+4}}{5}-1\,\,\,\ (\mbox{or} \, \,\, A+B \leq \frac{\sqrt{p+4}}{5}-1)$$
\end{remark}

In order to motivate the theorems below some comments are in order here. Notice that for a $p\in \mathbb{N}$ ending in $1$ the quadratic equation in two variables which represents such a number  can be somewhat approximated by another  quadratic equation which represents a natural number $p+101$. This number can be factorized easily since $1$ is a small natural number. The factors of  $p+101$ would be used as upper  and lower bounds in order to get the solutions of the quadratic equation representing the number $p$.
  
\begin{theorem}\label{t1}																							
Let $p$ be a natural number ending in 1 and multiple of $\overset{\circ}{3}+2$ or $\overset{\circ}{3}+1$. If $(A,B)\in\mathbb{N}\times \mathbb{N}$ satisfy the equation $p=(10x+3)(10y+7)$, then there exist $\lambda, C, D\in\mathbb{N} $\, U \, $\{0\}$ such that
$$A=\frac{10\lambda-4-D-\sqrt{(D+4-10\lambda)^{2}-4(1-3\lambda)}}{2}$$
and
$$B=\frac{10\lambda-4-D+\sqrt{(D+4-10\lambda)^{2}-4(1-3\lambda)}}{2},$$
where $A<B$ and $\frac{D}{7}-\frac{101^{2}p}{7.10^{6}}+\frac{33}{7}\leq \lambda \leq  \frac{D}{10}+\frac{3}{10}$ for all $A\geq 30$ and $B\geq 70$.
\end{theorem}
\begin{proof}
The proof is presented only for the case  $\overset{\circ}{3}+2$, the case $\overset{\circ}{3}+1$ follows similar steps.
Since $p=\overset{\circ}{3}+2$ then $p+10=\overset{\circ}{3}$ and so $p+10=(10C+3)(10D+7)$. It is important to emphasize that this is not the only representation of $p+10$ but the representation considered will allow us to obtain the desired upper and lower bounds. Then we consider $C=0$, since for this case on has $p+10=\overset{\circ}{3}$, which implies $D=\frac{p-11}{30}$. Through the following equations
\begin{equation*}
\begin{split}
p& = (10A+3)(10B+7)\\
p+10& = (10C+3)(10D+7)
\end{split}
\end{equation*}
we obtain $1=-10(AB+A)+3(D-B+A)$ which implies
\begin{equation}\label{e1}
    10(1+AB+A)=3(D-B+A+3)
\end{equation}
Using  (\ref{e1}) one has
\begin{equation}\label{c1}
\begin{split}
D-B+A+3& = 10\lambda\\
1+AB+A& = 3\lambda
\end{split}
\end{equation}
From the first equation of $(\ref{c1})$, we have
\begin{equation}\label{e2}
    \lambda\leq \frac{D}{10}+\frac{3}{10}
\end{equation}
Using the Theorem 2 of \cite{LS} one has
\begin{equation}\label{e3}
    \frac{33}{7}+\frac{D}{7}-\frac{101^{2}p}{7.10^{6}}\leq \lambda \,\,\,\mbox{for all}\,\, \,A\geq 30 \,\,\mbox{and}\,\, B\geq 70.
\end{equation}
Therefore from  (\ref{e2}) and (\ref{e3}) we obtain the desired inequality.
\end{proof} 
\begin{example}
Consider the number $p=900071$. So,  $p+10=900081=(10C+3)(10D+7)$, then $C=0$ which implies $D=\frac{p-11}{30}=3002$. Thus
\begin{equation}\label{eex1}
    2980\leq \lambda \leq 3000
\end{equation}
From the equations of $(\ref{c1})$ we conclude that $\lambda \neq \overset{\circ}{3}$. Then from $(\ref{eex1})$ $\lambda$ takes  thirteen possible values.
By replacing these possible values of $\lambda$ we can get $(A, B)\in\mathbb{N} \times \mathbb{N}$; therefore these $(A,B)$ would be the set of integer solutions of the quadratic polynomial that represents $p=900071$. In the case that there is no integer solution, that is, $A\notin \mathbb{N},\, B\notin \mathbb{N}$, then the twenty possible values of $A$; such that $A\leq 30$ ($A \neq \overset{\circ}{3}$) are replaced in the equation $900071=(10A+3)(10B+7)$.
\end{example}
\begin{theorem}\label{t2}																							
Let $p$ be a natural number ending in 1, multiple of $\overset{\circ}{3}+2$ (or  $\overset{\circ}{3}+1$) and $p+10=\overset{\circ}{9}$. If $(A,B)\in\mathbb{N}\times \mathbb{N}$ satisfy the equation $p=(10x+9)(10y+9)$, then there exist $\lambda, C, D\in\mathbb{N} \mbox{U} \{0\}$ such that
$$A=\frac{D+1-10\lambda-\sqrt{(10\lambda-1-D)^{2}-4(9\lambda-1)}}{2}$$
and
$$B=\frac{D+1-10\lambda+\sqrt{(10\lambda-1-D)^{2}-4(9\lambda-1)}}{2},$$
where $A<B$ and $D+1-\frac{101^{2}p}{10^{6}}\leq \lambda \leq  \frac{D+1}{10}+\frac{9-\sqrt{p}}{50}$ for all $A\geq 10$ and $B\geq 10$.
\end{theorem}
\begin{proof}
We will consider the case $\overset{\circ}{3}+2$. Since $p=\overset{\circ}{3}+2$ then $p+10=\overset{\circ}{3}$ and so $p+10=(10C+9)(10D+9)$. We consider $C=0$ since $p+10=\overset{\circ}{3}$, which implies $D=\frac{p-71}{90}$. Through the following equations
\begin{equation*}
\begin{split}
p& = (10A+9)(10B+9)\\
p+10& = (10C+9)(10D+9)
\end{split}
\end{equation*}
we obtain $1=-10AB+9(D-B-A)$ which implies
\begin{equation}\label{e4}
    10(1+AB)=9(D-B-A+1)
\end{equation}
From (\ref{e4}) one gets
\begin{equation*}
\begin{split}
D-B-A+1& = 10\lambda\\
1+AB& = 9\lambda
\end{split}
\end{equation*}
From  the first equation  and  $A+B\geq\frac{\sqrt{p}-9}{5}$ we have
\begin{equation}\label{e5}
    \lambda\leq \frac{D}{10}-\frac{1}{10}+\frac{9-\sqrt{p}}{50}.
\end{equation}
Using the theorem 2 of \cite{LS}  one gets
\begin{equation}\label{e6}
    D+1-\frac{101^{2}p}{10^{6}} \leq \lambda \,\,\,\mbox{for all}\,\, \,A\geq 10 \,\,\mbox{and}\,\, B\geq 10.
\end{equation}
Therefore from (\ref{e5}) and (\ref{e6}), we obtain the desired inequality.
\end{proof}
\begin{example}
Consider the number $p=900071$. Since $\overset{\circ}{3}=p+10=900081=(10C+9)(10D+9)$ and $900081=\overset{\circ}{9}$, then it is enough to consider $C=0$, which implies $D=\frac{p-71}{90}=10000$. Thus
\begin{equation}\label{eex2}
    820\leq \lambda \leq 981
\end{equation}
Using the following relationships $D-B-A+1=10\lambda$ and $1+AB=9\lambda$, we get $\lambda=\overset{\circ}{3}+2$. Then $\lambda$ takes $53$ possible values.
From the possible values of $\lambda$ such that  $(A,B)\in\mathbb{N} \times \mathbb{N}$; therefore $(A,B)$ would be an integer solution of the quadratic polynomial that represents $p=900071$. In the case that there is no such $\lambda$,  then the values $A\leq 10$ are replaced into the equation $900071=(10A+9)(10B+9)$ and as $A \neq \overset{\circ}{3}$, then $A$ takes $7$ possible values.
\end{example}
\begin{theorem}\label{t3}																				
Let $p$ be a natural number ending in 1, multiple of $\overset{\circ}{3}+2$ (or $\overset{\circ}{3}+1$)  and $p> 11$. If $(A,B)\in\mathbb{N}\times \mathbb{N}$ satisfy the equation $p=(10x+1)(10y+1)$, then there exist $\lambda, C, D\in\mathbb{N} \mbox{U} \{0\}$ such that
$$A=\frac{10\lambda+D-4-\sqrt{(40\lambda+D-4)^{2}-4(5+D-9\lambda)}}{2}$$
and
$$B=\frac{10\lambda+D-4+\sqrt{(40\lambda+D-4)^{2}-4(5+D-9\lambda)}}{2},$$
where $A<B$, $A\geq 90$, $B\geq 90$ and $D=\frac{p+10 L-11}{110}$ for $L\in \{1,2,\dots ,10\}$, and 
$2D+1-\frac{101^{2}p}{10^{6}}\leq \lambda \leq  \frac{D}{10}-\frac{1-\sqrt{p}}{10} - \frac{1}{2}.$
\end{theorem}
\begin{proof}
Since $p> 11$ it is easy to see that $p=\overset{\circ}{11}+k$, where $k\in \{1,2,\dots ,10\}$. We would like $p+10L$ to be a multiple of $\overset{\circ}{11}$ for some $L\in\mathbb{N}$. Then $p+10 L=\overset{\circ}{11}+k+10 L$ and in this case if $k=1$ then $L=1$; if $k=2$ then $L=2$ and so on. The last one becomes,  $k=10$ then $L=10$. Then we can assume that $L=5$. So $C=1$ and $D=\frac{p+39}{110}$.
From the next equations
\begin{equation*}
\begin{split}
p& = (10A+1)(10B+1)\\
p+10& = (10C+1)(10D+1)
\end{split}
\end{equation*}
we have
\begin{equation}\label{e7}
    10(5-AB+D)=9(5-D+A+B).
\end{equation}
From (\ref{e7}) one gets
\begin{equation}\label{e8}
\begin{split}
AB-D-5& = 9\lambda\\
D-A-B-5& = 10\lambda.
\end{split}
\end{equation}
Then from (\ref{e8}) and using the fact that $A+B\geq\frac{\sqrt{p}-1}{5}$, we have
\begin{equation}\label{e9}
    \lambda\leq \frac{D}{10}-\frac{(1-\sqrt{p})}{50}-\frac{1}{2}
\end{equation}
Using again the theorem 2 of \cite{LS} we get
\begin{equation}\label{e10}
    2D+1-\frac{101^{2}p}{10^{6}}\leq \lambda.
\end{equation}
Therefore, from (\ref{e9}) and (\ref{e10}), we obtain the desired inequality. 
Similarly we can proceed with the other values of L.
\end{proof}
\begin{center} \subsection*{Acknowledgment} \end{center}
The author thanks CONCYTEC for the  partial financial support and also thanks his family for giving him a nice  working environment. \\



\begin{thebibliography}{abc99}
\bibitem{LS}
B. M. Cerna Magui\~na, H. Blas and V. H. L\'opez Sol\'{\i}s, Some results on natural numbers represented by quadratic
polynomials in two variables, https://arxiv.org/abs/1808.06145
\bibitem{N}
I. Niven, H. S. Zuckerman and H. L. Montgomery (1991). An introduction to the theory of numbers. John Wiley and 
ns.
\bibitem{S}
P. Samuel (2013). Algebraic Theory of Numbers: Translated from the French by Allan J. Silberger. Courier Corporation.
\end{thebibliography}
\end{document}